\documentclass[a4paper,12pt,reqno]{amsart}
\usepackage{amsmath,amsfonts,amssymb,amsthm,enumerate,multicol}%hyperref,
\usepackage{tikz,graphicx}
\usepackage{stackengine}
\usepackage{subfigure,hyperref}
\usepackage{caption,enumitem,wasysym}
\usepackage{amsmath,amssymb}
\usepackage{cleveref}
\usepackage{mathrsfs}
\usepackage{physics}
\usepackage{graphics,graphicx}
\usepackage{stackengine}
\usepackage{bigints}
\usepackage{geometry}
\usepackage{epstopdf}
\usepackage{epsfig}
\usepackage{caption,subcaption}
\usepackage{mathtools}
\usepackage{ragged2e}
\usepackage[super]{nth}
\usepackage{float,wrapfig}
\usepackage[labelsep=space]{caption}
\usepackage[font=footnotesize]{caption}
\usepackage{xcolor}

\newtheorem{theorem}{Theorem}[section]
\newtheorem{definition}{Definition}[section]

\newtheorem{corollary}{Corollary}[section]
\newtheorem{remark}{Remark}[section]

\newtheorem{case}{Case}
\numberwithin{equation}{section}

\title{INTEGRAL REPRESENTATION AND FUNCTIONAL INEQUALITIES INVOLVING GENERALIZED POLYLOGARITHM}
\author{Deepshikha Mishra$^\dagger$}
\address{$^\dagger$Department of Mathematics\\ Indian Institute of Technology, Roorkee-247667, Uttarakhand, India}
\email{deepshikha\_m@ma.iitr.ac.in}
\author[A. Swaminathan]{A. Swaminathan\, $^{\ddagger}$}
\address{$^\ddagger$Department of Mathematics\\ Indian Institute of Technology, Roorkee-247667, Uttarakhand, India}
\email{mathswami@gmail.com, a.swaminathan@ma.iitr.ac.in}
\allowdisplaybreaks
\bigskip

\begin{document}
	
	
	\begin{abstract}
	The purpose of this manuscript is to derive two distinct integral representations of the generalized polylogarithm using two different techniques. The first approach involves the Dirichlet series and its Laplace representation, which leads to a single integral representation. The second approach utilizes the Hadamard convolution, resulting in a double integral representation. As a consequence, an integral representation of the Lerch transcendent function is obtained. Furthermore, we establish properties such as complete monotonicity, Turan inequality, convexity, and bounds of the generalized polylogarithm. Finally, we provide an alternative proof of an existing integral representation of the generalized polylogarithm using the Hadamard convolution.
	\end{abstract}
	
	%		This paper focuses on the generalized polylogarithm $\Phi_{p, q}(a, b; z)$, which extends the notion of classical polylogarithm. A new integral representation for $\Phi_{p, q}(a, b; z)$ is derived. Using this integral representation, we discuss its complete monotonicity along with related consequences and establish bounds for $\Phi_{p, q}(a, b; z)$. In the process, we modify an existing integral representation of the Lerch transcendent function $\Psi(z, s, a)$ to extend its applicability to a larger domain. Additionally, Turan-type inequalities for the generalized polylogarithm $\Phi_{p, q}(a, b; z)$ are explored. Finally, we present an alternative proof of the existing integral representation of the generalized polylogarithm $\Phi_{p, q}(a, b; z)$ using the integral form of Hadamard convolution.
	\maketitle
	\markboth{Deepshikha Mishra and A. Swaminathan}{Special functions}
	
	\section{Introduction}
	The formulation of numerous problems in physics and mathematics involves the theory of  Green's function. In some formulations, Green's function expands as a series of Clausen functions, and the solution involves multiple differentiations to turn these functions into logarithmic expressions, and eventual reintegration. To address this particular issue, Leibniz \cite{Leibniz_1962_polylogarithm} introduced the concept of polylogarithm.
	\begin{definition}\cite{Lewin_1981_polylogarithm and associated functions}
		The polylogarithm function or Jonquire's function, denoted as $Li_p(z)$   is defined by the infinite sum
		\begin{equation}\label{Polylogarithm equation}
			Li_p(z)=\sum_{n=0}^\infty \frac{z^n}{n^p},
		\end{equation}
		for all complex numbers $z$ satisfying $|z|< 1$ and $p \in \mathbb{C}$.
		However, for  $|z|=1$, $p \in \mathbb{C}$ satisfies $\Re (p) >1$.
	\end{definition}
	When $p =1$, \eqref{Polylogarithm equation} simplifies to $Li_1(z)=-\log(1-z)$, the well known logarithm function. For $ p = 2, 3 $, \eqref{Polylogarithm equation} reduces to $ Li_2(z) = \sum_{n=0}^\infty \frac{z^n}{n^2} $ and $ Li_3(z) = \sum_{n=0}^\infty \frac{z^n}{n^3} $, which are the well-known dilogarithm and trilogarithm functions, respectively. Lewin studied the behavior of dilogarithm and trilogarithm functions and provided a detailed theory of the polylogarithm in the monograph \textit{Polylogarithms and Associated Functions} \cite{Lewin_1981_polylogarithm and associated functions}. The theory of polylogarithm encompasses numerous properties, such as integral representations, functional relations, series expansions, linear and quadratic transformations, and numerical values for specific arguments. Integral representations of special functions have always been a topic of great interest because they provide alternative ways of expressing these functions and can be used to establish relationships between different special functions. Our primary objective in this discussion is to concentrate on the integral representation of the polylogarithm function. Several integral representations of polylogarithms are already available in the literature. Some fundamental integral forms are
	\begin{align*}
		Li_p(z)=\frac{z}{\Gamma(p)}\int_0^\infty \frac{t^{p-1}}{e^t-z}dt, \quad  |\arg(1-z)|<\pi, \  \Re(p)>0,
	\end{align*}
	and
	\begin{align*}
		Li_p(z)=\frac{z}{\Gamma(p)}\int_1^\infty \frac{\log^{p-1}t}{t(t-z)}dt=\frac{z}{\Gamma(p)}\int_0^1 \frac{{\log ^{p-1}(1/t)}}{1-zt}dt, \quad |\arg(1-z)|<\pi, \Re(p)>0,
	\end{align*}
	given in \cite{Prudnikov_1988_integral and series vol 2, Prudnikov_1990_integral and series vol 3}.
	To know more about polylogarithms, we refer to \cite{Lewin_polylogarithmic ladders, Lewin_1981_polylogarithm and associated functions} and references therein.
	As the theory of polylogarithms has evolved, numerous authors have studied it and introduced some of its generalizations. In \cite{Nielsen_1986_generalized polylogarithm}, Nielsen  introduced the first generalization of the polylogarithm function, referred to as the Nielsen generalized polylogarithm. Among this and various other generalizations, this manuscript focuses on a specific generalization of the polylogarithm, as outlined below.
	\begin{definition}\cite{Saiful_Swami_2010_geometric properties of polylogairthm}
		The generalized polylogarithm function, denoted as $\Phi_{p, q}(a,b;z)$, is defined by the following normalized expression
		\begin{equation}\label{generalized polylogarithm equation}
			\Phi_{p, q}(a,b;z) = \sum_{k=1}^\infty \frac{(1+a)^p(1+b)^q}{(k+a)^p(k+b)^q}z^k, \quad |z|<1,
		\end{equation}
		where $p$ and $q$ represent complex numbers with $\Re(p)>0$ and $\Re(q)>0$, while $a, b> 0$.
	\end{definition}
	This generalization given by \eqref{generalized polylogarithm equation} includes many well-known special functions as particular cases that are detailed below.
	\begin{enumerate}
		\item If $a=b,p+q= s$  then,
		\begin{align*}
			\Phi_{p, q}(a,b;z)=(1+a)^s \sum_{k=1}^{\infty} \frac{z^k}{(k+a)^s},
		\end{align*}
		 where the function
		\begin{align}\label{lerch transcendent equation}
			\Psi(z,s,a)= \sum_{n=0}^\infty \frac{z^n}{{(n+a)}^s}, \quad |z|<1,
		\end{align}
		with $a>0 , s \in \mathbb{C} $ such that $ \Re(s) > 0 $ is known as the \textit{Lerch transcendent function}, which serves as the generalization of both the \textit{Hurwitz zeta function} and the \textit{Riemann zeta function}.
		\item If $a=b=0,p+q=r \in \mathbb{C}, \Re(r)>0$ then $\Phi_{p, q}(a,b;z)=Li_r(z)$ which is again the polylogarithm function, given by \eqref{Polylogarithm equation}.
	\end{enumerate}
	The generalized polylogarithm \eqref{generalized polylogarithm equation} is  associated with the generalized hypergeometric function through the following relation \cite{Saiful_Swami_2010_geometric properties of polylogairthm}
	\begin{align*}
	\Phi_{p,q}(a,b;z)=z({}_{p+q+1}F_{p+q}(a_1,a_2,a_3,...,a_{p+q+1};c_1,c_2,c_3,...,c_{p+q})),
	\end{align*}
	where $p,q$ are natural numbers,
	\[a_i= 	\begin{cases}
		1,& \text{if } i=1\\
		1+a,  & \text{if } i=2,3,...,p+1\\
		1+b, & \text{if } i=p+2,p+3,...,p+q+1,
	\end{cases}\]
	and
	\[c_i= 	\begin{cases}
		
		2+a, & \text{if } i=1,2,3,...,p\\
		2+b, & \text{if } i=p+1,p+2,p+3,...,p+q.
	\end{cases}\]
Here ${}_pF_q(z)$ denotes the generalized hypergeometric function which is defined \cite{Rainville_1960_Special functions} as 
	\begin{align}\label{Generalized hypergeometric series}
	{}_pF_q(z) &= \sum_{k=0}^{\infty} \frac{(a_1)_k (a_2)_k \cdots (a_p)_k}{(b_1)_k (b_2)_k \cdots (b_q)_k} \frac{z^k}{k!},
\end{align}
where $(a)_k$ is the Pochhammer symbol given by
\begin{align*}
	(a)_k &=
	\begin{cases}
		1 & \text{if } k = 0, \\
		a(a+1)(a+2)\cdots(a+k-1) & \text{if } k > 0.
	\end{cases}
\end{align*}

In \eqref{Generalized hypergeometric series}, any denominator term cannot be zero or negative integer, and if any numerator parameter is zero, the series terminates. The convergence of \eqref{Generalized hypergeometric series} depends on $p$ and $q$ such that
\begin{enumerate}
	\item If $p \leq q$, the series converges for all finite $z$.
	\item If $p = q + 1$, the series converges for $|z| < 1$ and diverges for $|z| > 1$. For $|z| = 1$, absolute convergence occurs if
	\begin{align}
		\Re\left(\sum_{j=1}^q b_j - \sum_{i=1}^p a_i\right) &> 0.
	\end{align}
	\item If $p > q + 1$, the series converges only at $z = 0$.
\end{enumerate}

For $p = 2, q = 1$, ${}_pF_q(z)$ reduces to the Gauss hypergeometric function, and for $p = q = 1$, it simplifies to the confluent hypergeometric function. More details can be found in \cite{Andrews_1999_Special functions, Rainville_1960_Special functions}.
%	In addition to the aforementioned connection with the generalized hypergeometric function, there are alternative formulations that highlight the relationship between the generalized polylogarithm and the generalized hypergeometric function, interested one may refer to \cite{Lewin_polylogarithmic ladders, Saiful_Swami_2010_geometric properties of polylogairthm}.

In addition to the aforementioned connection with the generalized hypergeometric function, there are other significant contributions in the field of generalized polylogarithms. In \cite{Maximon_2003_the dilogarithm function for complex argument}, a comprehensive overview of the dilogarithm function and its integral forms was provided, highlighting its significance. Subsequently, the polylogarithm and its integral representations, which are closely related to Clausen functions, were examined in \cite{Cvijovic_2007_integral representation of polylogarithm}. Additionally, new integral expressions for the Mathieu series and the COM–Poisson normalization constant were introduced in \cite{Pogany_2005_integral representation of mathieu series,Pogany_2018_integral form of le roy function}. Inspired by these previous studies, this article primarily centres on expressing the generalized polylogarithm through integral representation.

\subsection{Motivation}
Integral representations of special functions are often more convenient for analyzing various properties of these functions, including their complete monotonicity, asymptotic behavior, inequalities, and analytic properties such as whether the function belongs to the class of Pick functions, Stieltjes functions, Bernstein functions, and so on. Nowadays, finding integral representations of special functions is gaining significant attention and is being studied extensively. For instance, in \cite{Berg_2024_A complete Bernstein function related to the fractal dimension of Pascal’s pyramid modulo a prime}, the integral representation of the function $\frac{\log(1+r z)}{\log(1+z)}$ is derived, which helps in determining the specific class to which the function belongs. Additionally, properties such as convexity and monotonicity are also explored. Further examples of the applications of integral representations can be found in \cite{Alzer_1998_Inequalities for the polygamma functions, Berg_2024_A complete Bernstein function related to the fractal dimension of Pascal’s pyramid modulo a prime, Koumandos_2015_On asymptotic expansions of generalized Stieltjes functions, Pedersen_2005_On the remainder in an asymptotic expansion of the double gamma function}.

\par
In this manuscript, we have derived two distinct integral representations of the generalized polylogarithm: one involving a single integral and the other involves a double integral. A detailed proof can be found in section \ref{Generalized polylogarithm as a single integral} and section \ref{Generalized polylogarithm as a double integral}.

%%%%%%%%%%%%%%%%%%%%%%%%%%%%%%%%%%%%%%%%%%%%%%%%%%%%%%%%%%%%%%%%%%%%%%%%%%

\subsection{Organization}
The structure of this work is as follows: In Section \ref{preliminaries section}, we begin by introducing the Dirichlet series and its Laplace representation, followed by definitions and theorems related to completely monotone functions and Hadamard convolution. We then define the Lerch transcendent function \eqref{lerch transcendent equation} and provide its one-parameter series representation. In Section \ref{Generalized polylogarithm as a single integral}, we focus on deriving the integral representation of the generalized polylogarithm $\Phi_{p, q}(a, b; z)$ involving only a single integral, as given in Theorem \ref{integral representation single integral theorem}. Furthermore, as a special case, we obtain the integral representation of the Lerch transcendent function. Moreover, we establish bounds for the generalized polylogarithm $\Phi_{p, q}(a, b; z)$ and explore the complete monotonicity and Turan inequality for $\psi(p)$ given in \eqref{value of psi(p)}. We also examine the Turan inequality for the generalized polylogarithm $\Phi_{p, q}(a, b; z)$ and conclude the section. Finally, in Section \ref{Generalized polylogarithm as a double integral}, we derive an integral representation of the generalized polylogarithm $\Phi_{p, q}(a, b; z)$ involving a double integral and conclude the section by re-establishing a previously known integral form of the generalized polylogarithm $\Phi_{p, q}(a, b; z)$.

%%%%%%%%%%%%%%%%%%%%%%%%%%%%%%%%%%%%%%%%%%%%%%%%%%%%%%%%%%%%%%%%%%%%%%%%%

\section{preliminaries}\label{preliminaries section}
This section provides the preliminary definitions and theorems that we use later to derive the integral form of the generalized polylogarithm. We begin by discussing the Dirichlet series and one of its most significant integral representation. 
%Next, we explain completely monotone functions and their characterization, known as the Bernstein theorem. We then present the Hadamard convolution and its integral form, which involves the product of two holomorphic functions, commonly referred to as the Hadamard product. Following that, we examine the Lerch transcendent function and its one-parameter series representation. The section concludes by exploring a well-known result concerning the integral form of the Lerch transcendent function.
\begin{definition}\cite{Hardy_1964_Dirichlet series}
A series of the form
	\begin{equation}\label{Dirichlet series}
		f(\xi) = \sum_{n=1}^\infty \alpha_n e^{-\beta_n \xi},
	\end{equation}
	is known as a Dirichlet series of type $\beta_n$, where $\beta_n$ is an increasing sequence of real numbers that diverges to infinity, $\alpha_n$ is a sequence of complex numbers, and $\xi$ is a complex variable.
\end{definition}
If $\beta_n = n$, the Dirichlet series \eqref{Dirichlet series} simplifies to $f(\xi) = \sum_{n=1}^\infty \alpha_n e^{-n \xi}$, which represents a power series in $e^{-\xi}$. Alternatively, if $\beta_n = \log n$, the series \eqref{Dirichlet series} transforms into
	\begin{align*}
		f(\xi) &= \sum_{n=1}^\infty \alpha_n n^{-\xi},
	\end{align*}
which is known as an ordinary Dirichlet series. Over time, Dirichlet series gained significant attention when the first attempt \cite{Cahen_1894_Laplace representation of dirichlet series} was made to construct a systematic theory of the function $f(\xi)$ defined in \eqref{Dirichlet series}, which became the starting point for most later research in related subjects \cite{Cahen_1894_Laplace representation of dirichlet series, Hardy_1964_Dirichlet series}. One of the important results discussed in \cite{Cahen_1894_Laplace representation of dirichlet series} is the Laplace integral representation for the Dirichlet series \eqref{Dirichlet series}, which can be written as
\begin{align}\label{Laplace representation of dirichlet series}
	f(\xi) &= \xi \int_0^\infty e^{-\xi t} \sum_{n: \beta_n \leq t} a_n  dt = \xi \int_0^\infty e^{-\xi t} \sum_{n=1}^{\lfloor\beta^{-1}(t)\rfloor} \alpha_n  dt, \quad \Re (\xi)>0,
\end{align}
where $\beta: \mathbb{R}^+ \to \mathbb{R}^+$ is a monotonically increasing function with a unique inverse $\beta^{-1}$, and $\beta|_{\mathbb{N}} = \beta_n$, and $\lfloor \cdot \rfloor$ denotes the greatest integer function. We use the representation \eqref{Laplace representation of dirichlet series} as a fundamental tool to derive the integral form of the generalized polylogarithms, which is in single integral form.

\begin{definition}\cite{Widder_1941_the laplace transform}\label{complete monotone function definition}
A function  $f(x)$ is said to be completely monotonic on an interval $(a,b)$ if it is infinitely differentiable, i.e., $f(x) \in C^\infty$, and satisfies the following condition
\begin{align*}
			(-1)^k f^{(k)}(x)\geq 0,
		\end{align*}
		for every non-negative integer $k$.
	\end{definition}
	In addition to its definition, complete monotonicity is frequently characterized by the Bernstein theorem, which is presented below.
	\begin{theorem}\cite{Schilling_2021_bernstein function}\label{Bernstein theorem}{(Bernstein theorem)}
		Let $f: (0,\infty) \mapsto \mathbb{R}$ be a completely monotonic function. Then  it can be expressed as the Laplace transform of a unique measure $\mu$ defined on the interval $[0,\infty)$. In other words, for every $x > 0$,
		\begin{align*}
			f(x)=\mathscr{L}\{(\mu; x)\}= \int_{[0,\infty)} e^{-xt} d\mu(t).
		\end{align*}	
		Conversely, if $\mathscr{L}(\mu; x)< \infty $ for every $x>0$, then the function $x \mapsto \mathscr{L}{(\mu; x)}$ is completely monotone.
	\end{theorem}
	The Bernstein theorem characterizes completely monotonic functions through their representation as Laplace transforms of unique measures. We now turn our attention to the Hadamard convolution, an important concept in complex analysis that relates to the multiplication of holomorphic functions. The Hadamard product facilitates the multiplication of two power series, as detailed in the following theorem.
	\begin{theorem}\cite{Titchmarsh_1958_theory of functions}\label{hadamard multiplication theorem}
		Let $f(z)$ and $g(z)$ be defined by the power series
		\begin{align*}
			f(z)=\sum_{n=0}^{\infty}f_n z^n, \quad g(z)=\sum_{n=0}^{\infty}g_n z^n.
		\end{align*}
		The Hadamard product of the functions $f(z)$ and $g(z)$ is defined as
		\begin{align}\label{hadamard product summation form}
			(f*g)(z)=\sum_{n=0}^\infty f_ng_nz^n,
		\end{align}
		which is well-defined for all $f, g\in\mathcal{H(\mathbb{D})}$.
		Here, $\mathcal{H(\mathbb{D})}$ is the set of all series $\sum_{n=0}^\infty a_nz^n$ defined on the open unit disc $\mathbb{D}$ with the property $\limsup_{n\to\infty}|a_n|^{1/n}\leq1$.
	\end{theorem}
	The Hadamard product \eqref{hadamard product summation form} has the integral representation \cite{Titchmarsh_1958_theory of functions}
	\begin{align*}
		(f*g)(z)=\frac{1}{2\pi i}\int_{|w|=r} f\left(\frac{z}{w}\right)g(w)\frac{dw}{w},\quad \quad |z|<r<1.
	\end{align*}
	This convolution of holomorphic functions is known as the \textit{Hadamard's multiplication theorem}.
%	\begin{definition}\cite{Abramowitz_1964_Handbook of mathematical functions}\label{Lerch transcendent function definition}
%		The Lerch transcendent function is defined as follows
%		\begin{align}\label{lerch transcendent equation}
%			\Psi(z,s,a)= \sum_{n=0}^\infty \frac{z^n}{{(n+a)}^s}, \quad |z|<1,
%		\end{align}
%		where $a, s \in \mathbb{C} $ such that $\Re(a) >0$ and $ \Re(s) > 0 $.
%	\end{definition}
\par

Now we explore the Lerch transcendent function, which is defined in \eqref{lerch transcendent equation}. We generalize the original series given in \eqref{lerch transcendent equation} by introducing an additional parameter $\lambda$ with the help of the fundamental integral form of the Lerch transcendent function, given by
\begin{align}\label{Lerch transcendent function integral form equation}
	\sum_{n=0}^{\infty} \frac{z^n}{(n+a)^s} = \frac{1}{a^s} + \frac{z}{\Gamma(s)} \int_{0}^{1} \left(\log\frac{1}{t}\right)^{s-1} \frac{t^a}{1-zt} \, dt, \quad |z| < 1,
\end{align}
with $a > 0, \, s \in \mathbb{C}$ such that $\Re(s) > 0$.

\begin{theorem}
Let $\Psi_{\lambda}(z, s, a)$ denote the generalization of the Lerch transcendent function. For $\lambda < \frac{1}{2}$ and $a, s \in \mathbb{C}$ with $\Re(a) > 0$ and $\Re(s) > 0$, the series representation is given by
\begin{align}\label{series representation of hurwitz zeta}
	\Psi_{\lambda}(z, s, a) = \sum_{n=0}^{\infty} \frac{1}{(1-\lambda)^{n+1}} \sum_{k=0}^{n} \binom{n}{k} (-\lambda)^{n-k} \frac{z^{k}}{(k+a)^{s}}, \quad |z| < 1.
\end{align}
\end{theorem}
	\begin{proof}
		Equation \eqref{Lerch transcendent function integral form equation}, can also be written as;
			\begin{align}\label{Lerch transcendent function integral without constant term}
			\sum_{n=0}^{\infty} \frac{z^n}{(n+a)^s} = \frac{1}{\Gamma(s)} \underbrace{\int_{0}^{1} \left(\log\frac{1}{t}\right)^{s-1} \frac{t^{a-1}}{1-zt}  dt}_{I}, \quad |z| < 1,
		\end{align}
		By applying the $\lambda$-method given in \cite{Alzer_2016_Series representations for special functions and mathematical constants}, we first obtain the series representation of the integral $I$. By substituting this series representation into \eqref{Lerch transcendent function integral without constant term}, we arrive at the final expression given in \eqref{series representation of hurwitz zeta}.
	\end{proof}
	\begin{remark}
		By setting $\lambda = 0$ in \eqref{series representation of hurwitz zeta}, we obtain the classical form of the series presented in \eqref{lerch transcendent equation}.
	\end{remark}

%%%%%%%%%%%%%%%%%%%%%%%%%%%%%%%%%%%%%%%%%%%%%%%%%%%%%%%%%%%%%%%%%%%%%%%%%%%%%%%%%%%%%%%	

\section{Generalized polylogarithm as a single integral}\label{Generalized polylogarithm as a single integral}

In this section, we derive the result concerning the integral representation of the generalized polylogarithm $\Phi_{p, q}(a, b; z)$. The simplicity of this representation lies in its dependence on a single integral. 
%Moreover, the integral representation we have obtained allows us to examine several properties of the function $\psi(p)$, including its complete monotonicity (which provides bounds for the generalized polylogarithm), log-convexity, and Turan inequality. The study of Turan inequality for $\psi(p)$ further motivates a discussion on the Turan inequality for the generalized polylogarithm, which we address and conclude in this section.

%%%%%%%%%%%%%%%%%%%%%%%%%%%%%%%%%%%%%%%%%%%%%%%%%%%%%%%%%%%%%%%%%%%%%%%%%%%%%

%If the condition
%\begin{align}\label{condition for q/p real}
%	\Re(p) \Im(q) - \Im(p) \Re(q) = 0
%\end{align}
%holds,

\begin{theorem}\label{integral representation single integral theorem}
Suppose $a$ and $b$ are real numbers with $a, b >1$, and let $p, q \in \mathbb{C}$ with $\Re(p) > 0, \Re(q) > 0$. If $ p $ and $ q$ lie in the same quadrant, then the generalized polylogarithm $ \Phi_{p, q}(a, b; z)$, as defined in \eqref{generalized polylogarithm equation}, satisfies the following

\begin{align}\label{integral representation single integral equation}
	\Phi_{p, q}(a, b; z) &=  (1 + a)^p (1 + b)^q \left( \frac{z}{1 - z} + \frac{pz}{z-1} \int_0^\infty e^{-p y} z^{j(y)}  dy \right),\quad |z| < 1,
\end{align}
where $\Delta(y)$ is an invertible function given as
\begin{align}\label{value of j(y) and delta}
	j(y) = \lfloor \Delta^{-1}(e^y) \rfloor, \quad \Delta(y) = (y+ a)(y+ b)^{q/p},
\end{align}
and $\lfloor \cdot \rfloor$ denotes the greatest integer function.
\end{theorem}

%%%%%%%%%%%%%%%%%%%%%%%%%%%%%%%%%%%%%%%%%%%%%%%%%%%%%%%%%%%%%%%%%%%%%%%%%%%%%%

\begin{proof}
	The generalized polylogarithm $\Phi_{p, q}(a,b;z)$, defined in \eqref{generalized polylogarithm equation}, can be expressed as
	\begin{align}\label{generalized polylog as dirichlet series}
		\Phi_{p, q}(a,b;z) &= (1+a)^p (1+b)^q \sum_{n=1}^\infty z^n e^{-p \log \left( (n + a)(n + b)^{q/p} \right)},
	\end{align}
	which turns out to be a Dirichlet series \eqref{Dirichlet series} with $\alpha_n = z^n$, $\beta_n = \log \left( (n + a)(n + b)^{q/p} \right)$, and $\xi = p$.
	
	Let $\beta(x) = \log \left( (x + a)(x + b)^{q/p} \right)$ for $x \in (0, \infty)$, which gives
	\begin{align*}
		e^{\beta(x)} = (x + a)(x + b)^{q/p}, \quad x \in (0, \infty).
	\end{align*}
	Differentiating with respect to $x$, we obtain
	\begin{align*}
		e^{\beta(x)} \beta^{'}(x) = \frac{q}{p}(x + a)(x + b)^{(q/p)-1} + (x + b)^{q/p}, \quad x \in (0, \infty).
	\end{align*}
	The condition of $p,q$  lying in the same quadrant implies that $\frac{q}{p} \in (0, \infty)$ and we also have $a, b > 1$, it follows that $e^{\beta(x)} \beta'(x)$ is always positive. Moreover, since $e^{\beta(x)}$ is always positive, $\beta'(x)$ must also be positive, which implies that $\beta(x)$ is a strictly increasing function. Consequently, $\beta(x)$ and $e^{\beta(x)}$ are invertible, as any continuous and monotone function has an inverse.
	
	Thus, by using the Laplace representation \eqref{Laplace representation of dirichlet series}, we can express the generalized polylogarithm \eqref{generalized polylog as dirichlet series} as follows
	\begin{align}\label{implicit form of n}
\nonumber		\Phi_{p, q}(a,b;z) &= p (1+a)^p (1+b)^q \int_0^\infty e^{-p y} \sum_{n:\beta_n \leq y} z^n \, dy, \\
		&= p (1+a)^p (1+b)^p \int_0^\infty e^{-p y} \sum_{n: (n + a)(n + b)^{q/p} \leq e^{y}} z^n \, dy.
	\end{align}
	Solving the inequality constraint $(n + a)(n + b)^{q/p} \leq e^{y}$ leads to
	\begin{align}\label{sum from n to j(y)}
	\Phi_{p, q}(a,b;z)= p (1+a)^p (1+b)^q \int_0^\infty e^{-p y} \sum_{n=1}^{j(y)} z^n \, dy,
	\end{align}
	where $j(y)$ is given in \eqref{value of j(y) and delta}.
	
	By summing the given geometric series in \eqref{sum from n to j(y)}, we obtain
	\begin{align*}
		\Phi_{p, q}(a,b;z) = p (1+a)^p (1+b)^q \int_0^\infty e^{-p y} \left( \frac{z}{1-z} \left( 1 - z^{j(y)} \right) \right) \, dy.
	\end{align*}
After completing the integration, the desired result is obtained, as shown in \eqref{integral representation single integral equation}.
\end{proof}

%%%%%%%%%%%%%%%%%%%%%%%%%%%%%%%%%%%%%%%%%%%%%%%%%%%%%%%%%%%%%%%%%%%%%%%%%%%%%

\begin{remark}
In \eqref{implicit form of n}, the summation starts at $1$ and continues up to a value $n$ that satisfies the inequality $(n + a)(n + b)^{q/p} \leq e^{y}$. Due to the presence of the factor $\frac{q}{p}$, the explicit range of summation cannot be determined directly, and is therefore expressed in terms of $j(y)$, where $j(y)$ is the inverse of the function $\Delta(y)$ defined in \eqref{value of j(y) and delta}. However, for certain values of $p$ and $q$, the explicit form of $j(y)$ can be determined, as shown below.
\end{remark}

\begin{corollary}
For $p = q$, the integral representation \eqref{integral representation single integral equation} becomes
\begin{align}\label{integral representation single integral equation for p=q}
	\Phi_{p, p}(a, b; z) &= (1 + a)^p (1 + b)^p \left( \frac{z}{1 - z} + \frac{pz}{z-1} \int_0^\infty e^{-p y} z^{n(y)}  dy \right), \quad |z|<1,
\end{align}
where $n(y)$ is as follows
\begin{align}\label{value of n(y)}
	n(y) = \Bigl\lfloor \frac{ -(a + b) + \sqrt{(a - b)^2 + 4 e^y}}{2}  \Bigr\rfloor.
\end{align}
\end{corollary}
\begin{proof}
Substituting $p = q$ into \eqref{implicit form of n}, we obtain
\begin{align}\label{implicit form of n for p=q}
\Phi_{p, p}(a, b; z) = p (1+a)^p (1+b)^p \int_0^\infty e^{-p y} \sum_{n: (n + a)(n + b) \leq e^{y}} z^n \, dy.
\end{align}
To determine the value of $n$, we solve the inequality $(n + a)(n + b) \leq e^{y}$. Rewriting it, we have
\begin{align*}
(n + a)(n + b) - e^{y} \leq 0.
\end{align*}
Define the function $f(x)$ as
\begin{align*}
f(x) = (x + a)(x + b) - e^{y}, \quad x \in (0, \infty).
\end{align*}
It is evident that $ f(x)$ is an increasing function with zeros at
\begin{align*}
x_1 = \frac{ -(a + b) - \sqrt{(a - b)^2 + 4 e^y}}{2}, \quad
x_2 = \frac{ -(a + b) + \sqrt{(a - b)^2 + 4 e^y}}{2}.
\end{align*}
Now, the inequality $f(x) \leq 0 $ holds only if
\begin{align*}
0 \leq x \leq \frac{ -(a + b) + \sqrt{(a - b)^2 + 4 e^y}}{2}.
\end{align*}
Restricting $x $ to natural numbers, we get
\begin{align*}
	1 \leq n \leq \Bigl \lfloor \frac{ -(a + b) + \sqrt{(a - b)^2 + 4 e^y}}{2}\Bigr \rfloor.
\end{align*}
Thus, \eqref{implicit form of n for p=q} becomes
\begin{align*}
\Phi_{p, p}(a, b; z) = p (1+a)^p (1+b)^p \int_0^\infty e^{-p y} \sum_{n=1}^{n(y)} z^n \, dy,
\end{align*}
where $n(y)$ is defined as in \eqref{value of n(y)}.

By summing the geometric series and carrying out the necessary calculations, we obtain the desired result, as shown in \eqref{integral representation single integral equation for p=q}.
\end{proof}

Another form of \eqref{integral representation single integral equation for p=q} can be obtained by substituting $e^y = t$ into \eqref{integral representation single integral equation}, as shown below
\begin{align*}
	\Phi_{p, p}(a, b; z) = (1+a)^p (1+b)^p \left(\frac{z}{1-z} + \frac{p z}{z-1} \int_1^\infty z^{\Bigl \lfloor \frac{-(a+b) + \sqrt{(a-b)^2 + 4 t}}{2} \Bigr \rfloor} \frac{dt}{t^{p+1}} \right).
\end{align*}
Now, we consider the case where two conditions are imposed simultaneously: $a = b$ and $p = q$. Under these constraints, the generalized polylogarithm reduces to the Lerch transcendent function $\psi(z, p, a)$, as defined in \eqref{lerch transcendent equation}. Consequently, under these same conditions, the integral representation of the generalized polylogarithm $\Phi_{p, q}(a, b; z)$ also reduces to the integral form of the Lerch transcendent function, as shown in the following corollary.

\begin{corollary}\label{integral form of Lerch transcendent function corollary}
For $a = b$ and $p = q$, the integral representation of $\Phi_{p, q}(a, b; z)$ given in \eqref{integral representation single integral equation} simplifies to the following form
\begin{align}\label{Lerch transcendent function integral form with infinite limit}
\sum_{n=0}^\infty \frac{z^n}{{(n+a)}^p} = \frac{1}{a^p} + \frac{z}{1-z} + \frac{pz}{z-1} \int_0^\infty e^{-p\tau} z^{\lfloor e^\tau - a \rfloor}  d\tau, \quad |z|<1.
\end{align}
\end{corollary}
\begin{proof}
Putting $a = b$ in \eqref{integral representation single integral equation for p=q}, we get
\begin{align}\label{integral form of generalized polylog for a=b and p=q}
	\Phi_{p, p}(a, a; z) = (1 + a)^{2p}  \left( \frac{z}{1 - z} + \frac{pz}{z-1} \int_0^\infty e^{-p y} z^{\lfloor e^{y/2}-a \rfloor}  dy \right),
\end{align}
From \eqref{generalized polylogarithm equation}, by substituting the value of $\Phi_{p, p}(a, a; z)$ into \eqref{integral form of generalized polylog for a=b and p=q}, we get

\begin{align*}
	\sum_{n=1}^{\infty} \frac{z^n}{(n+a)^{2p}}&= \frac{z}{1 - z} + \frac{pz}{z-1} \int_0^\infty e^{-p y} z^{\lfloor e^{y/2}-a \rfloor}  dy.\\
\implies	\sum_{n=1}^{\infty} \frac{z^n}{(n+a)^{p}}&= \frac{z}{1 - z} + \frac{pz}{2(z-1)} \int_0^\infty e^{-p y/2} z^{\lfloor e^{y/2}-a \rfloor}  dy.
\end{align*}
Adding $ \frac{1}{a^p}$ to both sides and substituting $ \frac{y}{2} = \tau $, we obtain \eqref{Lerch transcendent function integral form with infinite limit}.
\end{proof}

We have obtained the integral form of the generalized polylogarithm as well as the Lerch transcendent function. One common feature in both integral forms is that they both involve integrals that resemble the Laplace transform of some measure. This brings us to the motivation of discussing the complete monotonicity, as Bernstein's theorem states that a completely monotone function is the Laplace transform of a unique measure. In the next result, we will discuss the complete monotonicity of the function involving the generalized polylogarithm.
\begin{corollary}\label{completely monotonic function involving generalized polylogarithm corollary}
Let $ p, q $ be positive real numbers, and $ a, b > 1 $. Also, let $ x \in (0, 1) $ be a real number. Define the function
\begin{align}\label{value of psi(p)}
	\psi(p) = \left( \frac{\Phi_{p, q}(a, b; x)}{(1 + a)^p (1 + b)^q} - \frac{z}{1 - z} \right) \frac{x - 1}{p x},
\end{align}
where $\Phi_{p, q}(a, b; x)$ is the generalized polylogarithm defined in \eqref{generalized polylogarithm equation}. Then, the function $\psi(p)$ is completely monotonic with respect to $p$ on $(0, \infty)$.
\end{corollary}
\begin{proof}
From \eqref{integral representation single integral equation}, we observe that
\begin{align}\label{psi(p) in terms of integral}
	\int_0^\infty e^{-py} x^{j(y)} dy = \left( \frac{\Phi_{p, q}(a, b; x)}{(1 + a)^p (1 + b)^q} - \frac{x}{1 - x} \right) \frac{x - 1}{p x}
\end{align}
where $j(y)$ is the same as in \eqref{value of j(y) and delta}.
By comparing \eqref{psi(p) in terms of integral} and \eqref{value of psi(p)}, we get
\begin{align}\label{psi(p) integral form}
\psi(p)	= \int_0^\infty e^{-pt} x^{j(t)}  dt.
\end{align}
Thus, proving the complete monotonicity of $\psi(p)$ is equivalent to demonstrating the complete monotonicity of the integral given in \eqref{psi(p) integral form}.
Since the integral \eqref{psi(p) integral form} is a Laplace transform of a positive measure as $j(y)$ is always positive, Bernstein's theorem (see Theorem \ref{Bernstein theorem}) ensures its complete monotonicity with respect to $p$, implying that $\psi(p)$ is also completely monotonic.
\end{proof}
The complete monotonicity of the function $\psi(p)$ can also be seen graphically, as illustrated below.
\begin{figure}[H]
	\footnotesize
	\stackunder[5pt]{\includegraphics[scale=0.9]{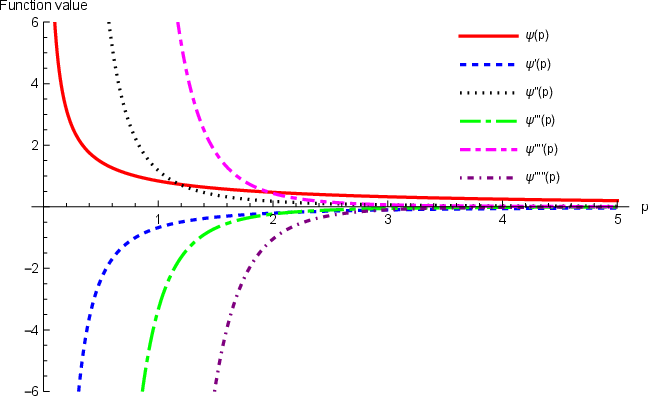}}{Graph of $\psi^{(n)}(p), n=0,1,2,3,4,5$ for $a=1.4$, $b=1.2$, $q=0.75$, $x=0.5$}
	\caption{Complete monotonicity of $\psi(p)$}
\label{fig:Complete monotonicity of psi function}
\end{figure}

The notation $\psi^{(n)}(p)$, as used in the footnote of Figure \ref{fig:Complete monotonicity of psi function}, denotes the $n$-th derivative of $\psi(p)$ with respect to $p$, where $n$ is a non-negative integer. From the graph, it is evident that $\psi(p)$ meets the criteria for complete monotonicity as outlined in Definition \ref{complete monotone function definition}. In this graphical representation, the derivatives of $\psi(p)$ are computed up to the fifth order, exhibiting an alternating sign pattern starting with a positive sign. However, the procedure can be extended to higher orders in a similar manner.

\begin{remark}
The function $\psi(p)$, as defined in \eqref{value of psi(p)}, is completely monotone on $(0, \infty)$. A significant subclass of completely monotone functions is the class of Stieltjes functions, leading to the question of whether $\psi(p)$ belongs to this subclass. We found that the function $\psi(p)$ is completely monotone, but it is not a Stieltjes function because the density function present in the integral given in \eqref{psi(p) integral form} is not completely monotone. It is not even differentiable and is only of bounded variation.
\end{remark}

Note that the complete monotonicity of $\psi(p)$ implies its non-negativity. By using this property, we can derive the bounds of the generalized polylogarithm, which are discussed in the following corollary.

\begin{corollary}
For positive real numbers $a, b, p$ and $q$, the bounds for the generalized polylogarithm $\Phi_{p, q}(a, b; x)$, as defined in \eqref{generalized polylogarithm equation}, are given by
\begin{align}\label{bound for generalized polylogarithm}
	x \leq \Phi_{p, q}(a, b; x) \leq (1 + a)^p (1 + b)^q \frac{x}{1 - x}, \quad x\in (0,1).
\end{align}
\end{corollary}
\begin{proof}
By Corollary \ref{completely monotonic function involving generalized polylogarithm corollary}, the non-negativity of $\psi(p)$ with respect to $p$ implies the bounds stated in \eqref{bound for generalized polylogarithm}.
\end{proof}
We now address the graphical representation of the bounds in \eqref{bound for generalized polylogarithm}.
\begin{figure}[H]
	\footnotesize
	\stackunder[5pt]{\includegraphics[scale=1.0]{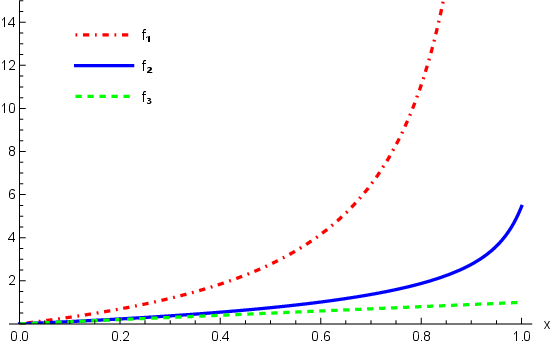}}{Graph of $f_1=(1 + a)^{p}(1 + b)^{q}\frac{x}{1 - x}$, $f_2=\Phi_{p, q}(a, b; x)$ and $f_3=x$ for $a=1.2, b=1.3, p=0.5, q=0.75$}
	\caption{Bounds for generalized polylogarithm $\Phi_{p, q}(a, b; x)$.}
	\label{fig:bounds for generalized polylogarithm}
\end{figure}

It is evident from the Figure \ref{fig:bounds for generalized polylogarithm} that the red (dot-dashed) curve, representing $f_1=(1 + a)^p (1 + b)^q \frac{x}{1 - x}$, provides an upper bound for the blue (thick) curve, which corresponds to the generalized polylogarithm $f_2=\Phi_{p, q}(a, b; x)$ for the given parameter values. Here the upper bound of $\Phi_{p, q}(a, b; x)$ depends on the specific values of the parameters $a$, $b$, $p$, and $q$. As these parameters vary, the upper bound adjusts accordingly, indicating that the bound is not uniform across different parameter settings. In contrast, the lower bound of $\Phi_{p, q}(a, b; x)$, represented by the green (dashed) curve, is $x$. This lower bound is independent of the parameters, indicating that it remains uniform regardless of the values of $a$, $b$, $p$, and $q$.
\par

Note that a completely monotonic function is always log-convex, and log-convexity is closely related to the Turan inequality, so the following result addresses the Turan-type inequalities for $\psi(p)$ as defined in \eqref{value of psi(p)}. These type inequalities often arise in the analysis of power series and special functions, including hypergeometric and polylogarithmic functions. Turan inequality were initially introduced for Legendre polynomials in \cite{Turan_1950_origination of turan inequality}. Since then, this inequality has been extensively studied and developed in various directions. For more literature on Turan inequality for various special functions, we refer readers to \cite{Baricz_2007_turan type inequalities for generalized complete elliptic integrals,Baricz_2008_turan inequalities for hypergeometric functions,Swami_Raghav__2013_Turan inequality} and the references therein.

\begin{corollary}
Let the function $\psi(p)$ be defined as in \eqref{value of psi(p)}. For positive real numbers $ x, q, a, b$ with $x\in(0,1)$ and $a,b > 1$ the following Turan-type inequality hold for $\psi(p)$
\begin{align}\label{turan inequality for psi(p)}
	\psi(p) \psi(p+2) - \psi^2(p+1)  \geq 0, \qquad p \in (0, \infty),
\end{align}
\end{corollary}

\begin{proof}
The complete monotonicity of a function implies its logarithmic convexity (see \cite{Widder_1941_the laplace transform}), it follows that the function $\psi(p)$ is log-convex. Therefore, we have the following inequalities

\begin{align*}
	\log (\psi(\alpha p_1 +(1-\alpha )p_2)) & \leq \alpha \log \psi(p_1) + (1-\alpha) \log \psi(p_2), \quad \alpha \in (0,1), \\
	\psi(\alpha p_1 +(1-\alpha )p_2) & \leq \psi^\alpha(p_1) + \psi^{(1-\alpha)}(p_2), \quad \alpha \in (0,1).
\end{align*}

Now, if we set $\alpha = 1/2$ and $p_2 = p_1+2$, we obtain the inequality \eqref{turan inequality for psi(p)}.
\end{proof}

The log-convexity of $\psi(p)$ and the Turan inequality of $\psi(p)$ are visually demonstrated in the figures below. Figure \ref{log convexity of psi function} illustrates the log-convexity of $\psi(p)$, while Figure \ref{Turan inequality of psi function} depicts the Turan inequality satisfied by $\psi(p)$.

\begin{figure}[H]
	\centering
	\begin{minipage}{0.45\textwidth}
		\centering
		\includegraphics[scale=0.7]{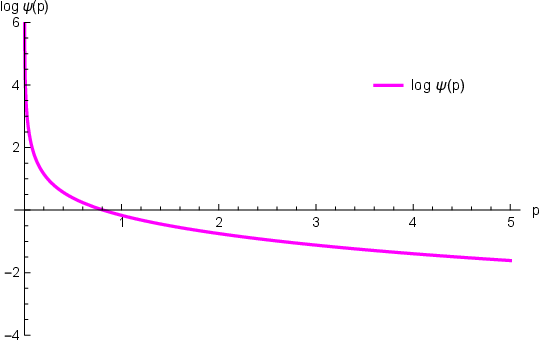}
		\footnotesize
		Graph of $\log(\psi(p))$ for $a=1.5,b=1.2,x=0.5,q=0.75$.
		\caption{\centering{Convexity of $\log \psi(p)$.}}
		\label{log convexity of psi function}
	\end{minipage} \hfill
	\begin{minipage}{0.45\textwidth}
		\centering
		\includegraphics[scale=0.7]{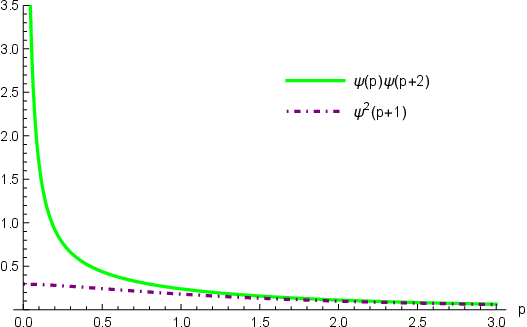}
		\footnotesize
		Graph of $\psi(p) \psi(p+2)$ and $\psi^2(p+1)$ for $a=1.6,b=1.2,x=0.5,q=0.75$.
		\caption{\centering{Turan inequality for $\psi(p)$.}}
		\label{Turan inequality of psi function}
	\end{minipage}
\end{figure}
In Figure \ref{log convexity of psi function}, the graph of $\log \psi(p)$ clearly illustrates its convexity with respect to $p$ over the interval $(0, \infty)$. Similarly, in Figure \ref{Turan inequality of psi function}, Turan's inequality is visually demonstrated, where the green (thick) curve, representing $\psi(p) \psi(p+2)$, consistently lies above the purple (dot-dashed) curve representing $\psi^2(p+1)$, as stated in \eqref{turan inequality for psi(p)}.

The expressions for $\psi(p)$ in \eqref{value of psi(p)} involve the generalized polylogarithm $\Phi_{p,q}(a,b;x)$. Also, we prove that $\psi(p)$ satisfy Turan's inequality. In line with this result, it is natural to ask whether the generalized polylogarithm $\Phi_{p,q}(a,b;x)$ satisfies Turan's inequality. We affirmatively address this question, and demonstrate the result in Theorem \ref{turan inequality of generalized polylog theorem}.

\begin{theorem}\label{turan inequality of generalized polylog theorem}
	Consider the generalized polylogarithm $\Phi_{p,q}(a,b;x)$ as defined in \eqref{generalized polylogarithm equation}. For positive real parameters $a$, $b$, $p$, and $q$, and for $x \in (0,1)$, the following inequalities hold
	\begin{align}\label{turan for generalized polylog for q}
		\Phi_{p,q}(a,b;x)\Phi_{p,q+2}(a,b;x) - \Phi_{p,q+1}^2(a,b;x) \geq 0,
	\end{align}
	and
	\begin{align}\label{turan for generalized polylog for p}
		\Phi_{p,q}(a,b;x)\Phi_{p+2,q}(a,b;x) - \Phi_{p+1,q}^2(a,b;x) \geq 0.
	\end{align}
\end{theorem}
\begin{proof}
	We begin with the expression for the generalized polylogarithm $\Phi_{p, q}(a,b;x)$ as defined in \eqref{generalized polylogarithm equation}, which can be rewritten as
	\begin{align}\label{modified generalized polylog}
		\Phi_{p, q}(a,b;x)= z \widetilde{\Phi}_{p, q}(a,b;x),
	\end{align}
	where
	\begin{align*}
		\widetilde{\Phi}_{p, q}(a,b;x)= \sum_{k=0}^\infty \frac{(1+a)^p(1+b)^q}{(k+1+a)^p(k+1+b)^q}x^k.
	\end{align*}
	Now, consider the product
	\begin{align*}
		\widetilde{\Phi}_{p, q}(a,b;x) \widetilde{\Phi}_{p, q+2}&(a,b;x) = \\
		&\sum_{k=0}^\infty \frac{(1+a)^p(1+b)^q}{(k+1+a)^p(k+1+b)^q}x^k \sum_{k=0}^\infty \frac{(1+a)^p(1+b)^{q+2}}{(k+1+a)^p(k+1+b)^{q+2}}x^k.
	\end{align*}
	By applying the Cauchy product formula, we obtain
	\begin{align*}
		\widetilde{\Phi}_{p, q}(a,b;x) &\widetilde{\Phi}_{p, q+2}(a,b;x) =\\
		&\sum_{n=0}^{\infty} \sum_{k=0}^{n} \frac{(1+a)^{2p} (1+b)^{2q+2} }{(k+1+a)^p(n-k+1+a)^p(k+1+b)^q(n-k+1+b)^{q+2}} x^n,
	\end{align*}
	and
	\begin{align*}
		\widetilde{\Phi}_{p,q+1}^2&(a,b;x) =\\
		&\sum_{n=0}^{\infty} \sum_{k=0}^{n} \frac{(1+a)^{2p} (1+b)^{2q+2}}{(k+1+a)^p(n-k+1+a)^p(k+1+b)^{q+1}(n-k+1+b)^{q+1}} x^n.
	\end{align*}
	Thus, we can write
	\begin{equation}\label{turan inequality of generalized polylog equation}
		\begin{aligned}
			\widetilde{\Phi}_{p,q}(a,b;x)\widetilde{\Phi}_{p,q+2}(a,b;x) - & \widetilde{\Phi}_{p,q+1}^2(a,b;x) =\\
			&(1+a)^{2p} (1+b)^{2q+2} \sum_{n=0}^{\infty} \sum_{k=0}^{n} f(n,k) x^n,
		\end{aligned}
	\end{equation}
	where
	\begin{align}\label{value of f(n,k)}
		f(n,k)= \frac{2k-n}{(k+1+a)^p(n-k+1+a)^p(k+1+b)^{q+1}(n-k+1+b)^{q+2}}.
	\end{align}
	We now analyze the finite sum $\sum_{k=0}^{n} f(n,k) $.
	\setcounter{case}{0}
	\begin{case}
		If $n$ is even.
	\end{case}
	In this case, we can express $\sum_{k=0}^{n} f(n,k) $ as follows
	\begin{align*}
		\sum_{k=0}^{n} f(n,k) &= \sum_{k=0}^{n/2-1} f(n,k) + \sum_{k=n/2+1}^{n} f(n,k) + f(n,n/2) \\
		&= \sum_{k=0}^{n/2-1} f(n,k) + \sum_{k=0}^{n/2-1} f(n,n-k) + f(n,n/2),
	\end{align*}
	Since $f(n,n/2)$ equals to $0$ from \eqref{value of f(n,k)} therefore, $\sum_{k=0}^{n} f(n,k)$ becomes
	\begin{align}\label{f(n,k) for even n}
		\sum_{k=0}^{n} f(n,k) = \sum_{k=0}^{\lfloor(n-1)/2\rfloor} \left(f(n,k) + f(n,n-k)\right),
	\end{align}
	where $\lfloor \cdot \rfloor$ denotes the greatest integer function.
	\begin{case}
		If $n$ is odd.
	\end{case}
	If $n$ is odd, then we have
	\begin{align*}
		\sum_{k=0}^{n} f(n,k) &= \sum_{k=0}^{(n-1)/2} f(n,k) + \sum_{k=(n+1)/2}^{n} f(n,k).
	\end{align*}
	Combining this into a single summation, we can express it as
	\begin{align}\label{f(n,k) for odd n}
		\sum_{k=0}^{n} f(n,k) = \sum_{k=0}^{\lfloor(n-1)/2\rfloor} \left(f(n,k) + f(n,n-k)\right),
	\end{align}
	where $\lfloor \cdot \rfloor$ is the greatest integer function. Now we observe that
	\begin{align*}
		f(n,k) + f(n,n-k) = \frac{(2k - n)^2}{(k+1+a)^p (n-k+a)^p (k+1+b)^{q+2} (n-k+b)^{q+2}} \geq 0.
	\end{align*}
	From \eqref{f(n,k) for even n} and \eqref{f(n,k) for odd n}, we obtain
	\begin{align*}
		\sum_{k=0}^{n} \left(f(n,k) + f(n,n-k)\right) \geq 0.
	\end{align*}
	Hence, using \eqref{turan inequality of generalized polylog equation}, we conclude that
	\begin{align}\label{turan type of modified generalized polylog}
		\widetilde{\Phi}_{p,q}(a,b;x) \widetilde{\Phi}_{p,q+2}(a,b;x) - \widetilde{\Phi}_{p,q+1}^2(a,b;x) \geq 0.
	\end{align}
	Now from \eqref{modified generalized polylog}, we have the following relationship
	\begin{align}\label{relation between phi and phi tilde}
		\Phi_{p,q}(a,b;x)\Phi_{p,q+2}(a,b;x) &- \Phi_{p,q+1}^2(a,b;x) = \\
		& x^2 \left(\widetilde{\Phi}_{p,q}(a,b;x)\widetilde{\Phi}_{p,q+2}(a,b;x) - \widetilde{\Phi}_{p,q+1}^2(a,b;x)\right).
	\end{align}
By applying \eqref{turan type of modified generalized polylog} in \eqref{relation between phi and phi tilde}, we assert that \eqref{turan for generalized polylog for q} is valid. In the same way, we can prove the result \eqref{turan for generalized polylog for p}.
\end{proof}
We now provide a graphical interpretation of the  inequality for the generalized polylogarithm $\Phi_{p,q}(a,b;x)$.
\begin{figure}[H]
	\footnotesize
	\stackunder[5pt]{\includegraphics[scale=1.0]{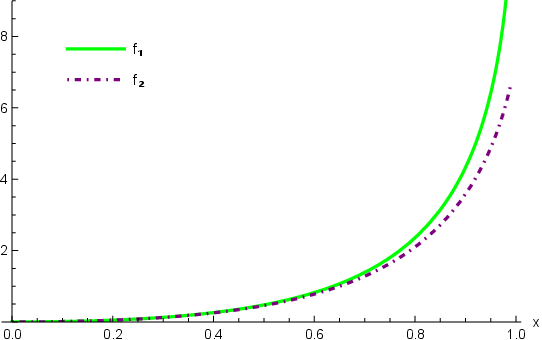}}{Graph of $f_1=\Phi_{p,q}(a,b;x)\Phi_{p,q+2}(a,b;x)$ and $f_2=\Phi_{p,q+1}^2(a,b;x)$ for $p=0.5, q=0.75, a=1.4,  b = 1.8$.}
	\caption{Turan inequality for the $\Phi_{p,q}(a,b;x)$ with respect to $q$}
	\label{fig:Turan inequality of the generalized polylogarithm}
\end{figure}

In the graphical interpretation, we plot the function $f_1=\Phi_{p,q}(a,b;x)\Phi_{p,q+2}(a,b;x)$, shown as the green (thick) curve and $f_2=\Phi_{p,q+1}^2(a,b;x)$ as the purple (dot-dashed) curve. With values $a = 1.4, b = 1.8, p=0.5$ and $q=0.75$, we observe that the green (thick) curve consistently bounds the purple (dot-dashed) curve, illustrating the Turan inequality for $\Phi_{p,q}(a,b;x)$ with respect to $q$, as established in \eqref{turan for generalized polylog for q}. Similarly, we can illustrate \eqref{turan for generalized polylog for p} graphically.

\section{Generalized polylogarithm as a double integral}\label{Generalized polylogarithm as a double integral}

In this section, we derive the integral representation of the generalized polylogarithm, which involves a double integral. To derive the integral form, we begin by decomposing the series of the generalized polylogarithm into the series of Lerch transcendent function using the Hadamard convolution. Subsequently, by applying the integral form of the Lerch transcendent function and solving certain complex integrals, we obtain the result. We conclude the section by demonstrating the usefulness of the Hadamard convolution, presenting another integral form of the generalized polylogarithm.

\begin{theorem}\label{Integral form of the generalized polylogarithm theorem}
Let $p, q$ be complex numbers having $\Re(p) > 0, \Re(q) > 0$, and $a, b$ be any real numbers such that $a, b >1$. The generalized polylogarithm function $ \Phi_{p, q}(a, b; z)$, defined as in \eqref{generalized polylogarithm equation}, admits the following integral representation
\begin{equation}\label{integral form of the generalized polylog as a double integral}
	\Phi_{p, q}(a,b;z) =
	(1+a)^p(1+b)^q \left(\frac{zq}{\Gamma (p)}\int_{0}^{\infty} \int_{0}^{\infty} e^{-qu} e^{-(a+1)v} v^{p-1} \frac{1-(ze^{-v})^{\lfloor e^u-b\rfloor}}{1-ze^{-v}} du dv\right)
\end{equation}
for all $|z|<1$ where $\lfloor \cdot \rfloor$ denotes the greatest integer function.
\end{theorem}
\begin{proof}
The generalized polylogarithm \eqref{generalized polylogarithm equation} can be written as
\begin{align}\label{generalized polylogarithm from n=0}
\Phi_{p, q}(a,b;z) = \sum_{k=0}^\infty \frac{(1+a)^p(1+b)^q}{(k+a)^p (k+b)^q}z^k - \frac{(1+a)^p(1+b)^q}{a^p b^q}.
\end{align}
If we take
\begin{align}\label{value of f_k and g_k}
f_k = \frac{1}{{(k+a)}^p}, \quad g_k = \frac{1}{{(k+b)}^q},
\end{align}
then the expression \eqref{generalized polylogarithm from n=0} becomes
\begin{align}\label{polylog eq modified}
\Phi_{p, q}(a,b;z) = (1+a)^p(1+b)^q \sum_{k=0}^\infty f_k g_k z^k - \frac{(1+a)^p(1+b)^q}{a^p b^q}.
\end{align}
Now, we define the functions $f(z)$ and $g(z)$ as
\begin{align}\label{f(z) and g(z) functions}
f(z) = \sum_{k=0}^\infty f_k z^k, \quad g(z)= \sum_{k=0}^\infty g_k z^k, \quad |z| < 1,
\end{align}
where $f_k$ and $g_k$ are given in \eqref{value of f_k and g_k}.

It is evident that $f(z)$ and $g(z)$ are holomorphic functions, with the property that
\begin{align*}
\limsup_{k \to \infty} |f_k|^{1/k} \leq 1, \quad \limsup_{k \to \infty} |g_k|^{1/k} \leq 1.
\end{align*}
Thus, by applying Theorem \ref{hadamard multiplication theorem} of the Hadamard convolution of $f(z)$ and $g(z)$, we obtain the following result
\begin{align}\label{Hadamard integral for Lerch transcendent}
\sum_{k=0}^\infty f_k g_k z^k &= \frac{1}{2\pi i} \int_{|\omega| = r} \left( \sum_{k=0}^\infty \frac{1}{(k+a)^p} \left(\frac{z}{\omega}\right)^k \sum_{k=0}^\infty \frac{1}{(k+b)^q} \omega^k \right) \frac{d\omega}{\omega}, \quad |\omega| < r < 1.
\end{align}
Using \eqref{Lerch transcendent function integral form equation} and \eqref{Lerch transcendent function integral form with infinite limit} in \eqref{Hadamard integral for Lerch transcendent}, we get
\begin{align*}
\sum_{k=0}^\infty f_k g_k z^k &= \frac{1}{2\pi i} \int_{|\omega| = r} \left( \frac{1}{a^p} + \frac{z}{\Gamma(p)} \int_0^1 \left(\log \frac{1}{t}\right)^{p-1} \frac{t^a}{\omega - zt}  dt \right) \\
&\quad \times \left( \frac{1}{b^q} + \frac{\omega}{1 - \omega} + \frac{q \omega}{\omega - 1} \int_0^\infty e^{-q u} \omega^{\lfloor e^u - b \rfloor}  du \right) \frac{d\omega}{\omega}, \quad |z| < r < 1.
\end{align*}
By substituting the value of $\sum_{k=0}^\infty f_k g_k z^k$ into \eqref{polylog eq modified}, we obtain
\begin{equation}\label{generalized polylog biggest integral}
	\begin{aligned}
		\Phi_{p, q}(a, b; z)&=  \frac{(1+a)^p(1+b)^q}{2 \pi i} \Bigg[  \int_{|w|=r} \left(\frac{1}{a^p b^q} + \frac{1}{a^p} \frac{w}{1-w}\right) \frac{dw}{w} \\
		&+\frac{z}{\Gamma (p)} \int_{|w|=r}  \left( \frac{1}{b^q} + \frac{w}{1-w} \right)
		\left( \int_{0}^{1} \left(\log\frac{1}{t}\right)^{p-1} \frac{t^a}{w-zt}  dt \right) \frac{dw}{w} \\
		& + \frac{q}{a^p}\int_{|w|=r} \frac{w}{(w-1)}
		\left( \int_0^\infty e^{-qu} w^{\lfloor e^u-b \rfloor}  du \right) \frac{dw}{w} \\
		& +\frac{zq}{\Gamma (p)}  \int_{|w|=r} \left( \frac{w}{w-1} \right)
		\left( \int_0^\infty e^{-q u} w^{\lfloor e^u-b \rfloor}  du \right)
		\left( \int_{0}^{1} \left(\log\frac{1}{t}\right)^{p-1} \frac{t^a}{w-zt}  dt \right) \frac{dw}{w} \Bigg] \\
		& - \frac{(1+a)^p(1+b)^q}{a^p b^q}.
	\end{aligned}
\end{equation}
Rewriting \eqref{generalized polylog biggest integral} in terms of $I_1$, $I_2$, $I_3$, and $I_4$, we get the following expression
\begin{align}\label{generalized polylog in terms of I's}
	\Phi_{p, q}(a,b;z)=
	(1+a)^p(1+b)^q (I_1+I_2+I_3+I_4)-\frac{(1+a)^p(1+b)^q}{a^p b^q}.
\end{align}
where the $I_n$'s are defined as follows
\begin{equation*}
\begin{aligned}
	&I_1= \frac{1}{2\pi i}\int_{|w|=r} \left(\frac{1}{a^p b^q}+\frac{1}{a^p} \frac{w}{1-w}\right)\frac{dw}{w}.\\
	&I_2=\frac{1}{2\pi i}\int_{|w|=r} \frac{z}{\Gamma (p)} \left( \frac{1}{b^q}+\frac{w}{1-w} \right) \left(\int_{0}^{1} \left(\log\frac{1}{t}\right)^{p-1} \frac{t^a}{w-zt}  dt \right) \frac{dw}{w}.\\
	&I_3=\frac{1}{2\pi i}\int_{|w|=r} \frac{qw}{a^p(w-1)} \left( \int_0^\infty e^{-qu} w^{\lfloor e^u-b \rfloor} du \right) \frac{dw}{w}.\\
	&I_4 = \frac{1}{2\pi i}\int_{|w|=r} \left( \frac{zq w}{\Gamma (p) (w-1)} \right)\left( \int_0^\infty e^{-q u} w^{\lfloor e^u-b \rfloor} du \right) \left( \int_{0}^{1} \left(\log\frac{1}{t}\right)^{p-1} \frac{t^a}{w-zt}  dt \right) \frac{dw}{w}.		
\end{aligned}
\end{equation*}
Now, we first find the values of $I_1$, $I_2$, $I_3$, and $I_4$.

Since the complex integral $I_1$ is defined over a closed curve $|w|=r$, we apply the Cauchy integral formula, which gives us
\begin{equation}\label{I_1 integral value}
	I_1 = \frac{1}{a^p b^q}.
\end{equation}
The integral $I_2$ can be written as;
\begin{align*}
	I_2 &= \frac{z}{\Gamma (p)}\int_{0}^{1}  t^a  \left(\log\frac{1}{t}\right)^{p-1} \underbrace{\left(\frac{1}{2\pi i}\int_{|w|=r}  \left( \frac{1}{b^q (w-zt)}+\frac{w}{(1-w)(w-zt)} \right)  \frac{dw}{w}\right)}_C  dt.
\end{align*}
Again, we use the Cauchy integral formula to evaluate the integral $C$, and we obtain: $C=\frac{1}{1-zt}$.
Therefore, $I_2$ becomes
\begin{align}\label{I_2 integral value}
	I_2 = \frac{z}{\Gamma (p)}\int_{0}^{1}  \left(\log\frac{1}{t}\right)^{p-1} \frac{ t^a }{1-zt} dt.
\end{align}
In the same manner, we find the integrals $I_3$ and $I_4$, which are
\begin{align}\label{I_3 and I_4 integral value}
I_3=0, \quad I_4= \frac{zq}{\Gamma (p)}\int_{0}^{1} \int_0^\infty \left(\log\frac{1}{t}\right)^{p-1} e^{-q u} (zt)^{\lfloor e^u-b \rfloor}  \frac{t^a}{zt-1}  du dt.
\end{align}
Substituting the value of $I_1$ from \eqref{I_1 integral value}, $I_2$ from \eqref{I_2 integral value}, and $I_3$ and $I_4$ from \eqref{I_3 and I_4 integral value} into \eqref{generalized polylog in terms of I's}, we obtain the following expression.
\begin{equation}\label{generalized polylog single-double integral}
	\begin{aligned}
		\Phi_{p, q}(a,b;z)=&(1+a)^p(1+b)^q \Bigg(\frac{1}{a^pb^q} + \frac{z}{\Gamma (p)}\int_{0}^{1}  \left(\log\frac{1}{t}\right)^{p-1} \frac{ t^a }{1-zt} dt\\
		&+\frac{zq}{\Gamma (p)}\int_{0}^{1} \int_0^\infty \left(\log\frac{1}{t}\right)^{p-1} e^{-q u} (zt)^{\lfloor e^u-b \rfloor}  \frac{t^a}{zt-1}  du dt \Bigg)-\frac{(1+a)^p(1+b)^q}{a^p b^q}.\\
	\end{aligned}
\end{equation}
By applying simple computations to \eqref{generalized polylog single-double integral}, we arrive at		
\begin{equation*}
	\begin{aligned}	
		\Phi_{p, q}(a,b;z)= &  (1+a)^p(1+b)^q \Bigg(\frac{zq}{\Gamma (p)}\int_{0}^{1} \int_{0}^{\infty} e^{-qu}\left(\log\frac{1}{t}\right)^{p-1} \frac{ t^a }{1-zt} du dt\\
		& +\frac{zq}{\Gamma (p)}\int_{0}^{1} \int_0^\infty \left(\log\frac{1}{t}\right)^{p-1} e^{-q u} (zt)^{\lfloor e^u-b \rfloor}  \frac{t^a}{zt-1}  du dt \Bigg).\\
		= & (1+a)^p(1+b)^q \Bigg(\frac{zq}{\Gamma (p)}\int_{0}^{1} \int_{0}^{\infty} e^{-qu}\left(\log\frac{1}{t}\right)^{p-1} \frac{ t^a }{1-zt}(1-(zt)^{\lfloor e^u-b \rfloor}) du dt \Bigg).
	\end{aligned}
\end{equation*}
Now, substituting $t = e^{-v}$, the remaining part of the proof follows.
\end{proof}
We have obtained the integral representation of the generalized polylogarithm $\Phi_{p, q}(a, b; z)$, where the Hadamard convolution played a pivotal role in the derivation. Not only this, but by the same convolution, we can derive other integral forms as well. One such example is shown below, where we found another integral representation of the generalized polylogarithm.

So, from \eqref{Lerch transcendent function integral form equation}, we have
\begin{align}\label{hurwitz zeta integral form 1}
	\sum_{n=0}^{\infty} \frac{z^n}{(n+a)^p} = \frac{1}{a^p} + \frac{z}{\Gamma(p)} \int_{0}^{1} \left(\log\frac{1}{t}\right)^{p-1} \frac{t^a}{1-zt}  dt, \quad |z| < 1,
\end{align}
and
\begin{align}\label{hurwitz zeta integral form 2}
	\sum_{n=0}^{\infty} \frac{z^n}{(n+b)^q} = \frac{1}{b^q} + \frac{z}{\Gamma(q)} \int_{0}^{1} \left(\log\frac{1}{u}\right)^{q-1} \frac{u^b}{1-zu}  du, \quad |z| < 1,
\end{align}
where $a$, $b$, $p$, and $q$ are complex numbers satisfying $\Re(a) > 0$, $\Re(p) > 0$, $\Re(b) > 0$, and $\Re(q) > 0$.

We apply Theorem \ref{hadamard multiplication theorem} to the holomorphic functions given in \eqref{hurwitz zeta integral form 1} and \eqref{hurwitz zeta integral form 2}, which results in
\begin{align}\label{integral form of hadamard convolution of hurwitz_zeta}
	\sum_{n=0}^{\infty} \frac{z^n}{(n+a)^p(n+b)^q} = \frac{1}{2 \pi i} \int_{|w|=r} \left( \sum_{n=0}^{\infty} \frac{z^n}{(n+a)^p} \sum_{n=0}^{\infty} \frac{z^n}{(n+b)^q} \right) \frac{dw}{w}.
\end{align}
Substituting \eqref{hurwitz zeta integral form 1} and \eqref{hurwitz zeta integral form 2} into \eqref{integral form of hadamard convolution of hurwitz_zeta} gives
\begin{align*}
	\sum_{n=0}^{\infty} \frac{z^n}{(n+a)^p(n+b)^q} = \frac{1}{2 \pi i} &\int_{|w|=r} \left( \frac{1}{a^p} + \frac{z}{\Gamma(p)} \int_{0}^{1} \left(\log\frac{1}{t}\right)^{p-1} \frac{t^a}{1-zt}  dt \right) \\
	&\times \left( \frac{1}{b^q} + \frac{z}{\Gamma(q)} \int_{0}^{1} \left(\log\frac{1}{u}\right)^{q-1} \frac{u^b}{1-zu}  du \right) \frac{dw}{w}.
\end{align*}
Solving the above integral using complex integral formulas and applying simple computations, we obtain the expression
\begin{align}\label{pounnusamy integral form equation}
	\Phi_{p, q}(a, b; z) =  \frac{(1+a)^p (1+b)^q}{\Gamma(p) \Gamma(q)} \int_{0}^{1} \int_{0}^{1} z \left(\log \frac{1}{u}\right)^{p-1} \left(\log \frac{1}{v}\right)^{q-1} \frac{u^a v^b}{1 - uvz}  du  dv, \quad |z| < 1,
\end{align}
where $p$, $q$, $a$, and $b$ are complex numbers with positive real parts.

Thus, we have obtained another integral form. However, the representation \eqref{pounnusamy integral form equation} is already presented in \cite{Ponnusamy_1996_Polylogarithms in the theory of univalent functions} and has been widely used to study the behavior of the generalized polylogarithm $\Phi_{p, q}(a, b; z)$ in the complex plane, as noted in \cite{Saiful_Swami_2010_geometric properties of polylogairthm, Ponnusamy_1996_Polylogarithms in the theory of univalent functions}.

\textbf{Concluding Remarks:} This manuscript primarily aimed to derive integral representations of the generalized polylogarithm $\Phi_{p, q}(a, b; z)$. Since the generalized polylogarithm is a form of Dirichlet series, we applied the Laplace representation of Dirichlet series under certain constraints on the parameters $p$, $q$, $a$, and $b$. Using this approach, we derived the single integral form of the generalized polylogarithm in \eqref{integral representation single integral equation}. This integral form allowed us to explore several properties, including complete monotonicity, the Turan inequality, and the convexity of the generalized polylogarithm and related functions.

Subsequently, we derived another integral form \eqref{integral form of the generalized polylog as a double integral} where first we found the integral form of the Lerch transcendent function. After that by using  the Hadamard product and the integral form of Lerch transcendent function we obtain the integral form of the generalized polylogairthm.

In terms of comparison between the two integrals, \eqref{integral representation single integral equation} and \eqref{integral form of the generalized polylog as a double integral}, the integral in \eqref{integral form of the generalized polylog as a double integral} is valid over a larger domain compared to \eqref{integral representation single integral equation}. However, the representation in \eqref{integral representation single integral equation} enables us to discuss a wider range of analytical properties of the generalized polylogarithm and related functions.

\end{document}